\documentclass[a4paper,12pt,leqno]{amsart} 

\usepackage{amsmath, amsthm, amssymb}
\usepackage{eucal}

\theoremstyle{plain} 
\newtheorem{theorem}{Theorem}[section]
\newtheorem{lemma}[theorem]{Lemma}

\newtheorem{corollary}[theorem]{Corollary}

\theoremstyle{definition} 

\newtheorem{remark}[theorem]{Remark}

\numberwithin{equation}{section}

\newcommand{\skal}[2]{\langle #1,#2\rangle}

\newcommand{\tilh}{\tilde h}

\begin{document}
\baselineskip=17pt

\title[Differential as a harmonic morphism]{Differential as a harmonic morphism with respect to Cheeger--Gromoll type metrics}
\author{W. Koz\l owski \and K. Niedzia\l omski}
\date{\today}
\address{
Department of Mathematics and Computer Science \endgraf
University of \L\'{o}d\'{z} \endgraf
ul. Banacha 22, 90-238 \L\'{o}d\'{z} \endgraf
Poland
}
\email{wojciech@math.uni.lodz.pl}

\address{
Department of Mathematics and Computer Science \endgraf
University of \L\'{o}d\'{z} \endgraf
ul. Banacha 22, 90-238 \L\'{o}d\'{z} \endgraf
Poland
}
\email{kamiln@math.uni.lodz.pl}

\begin{abstract}
We investigate horizontal conformality of a differential of a map between Riemannian manifolds where the tangent bundles are equipped with Cheeger--Gromoll type metrics. As a corollary, we characterize the differential of a map as a harmonic morphism. 
\end{abstract}

\maketitle

\section{Introduction and preliminary results}
\subsection{Introduction}
In \cite{blw1}, M. Benyounes, E. Loubeau and C. M. Wood intoduced a new class of Riemannian metrics in the tangent bundle over a Riemanian manifold. These metrics $h_{p,q}$, depending on two constants $p,q$, generalize Sasaki and Cheeger--Gromoll metrics. Wider class of this type of metrics was investigated by M. I. Munteanu \cite{mun} and by the authors \cite{kn}. Now, there are several articles concerning geometry of a tangent bundle equipped with Cheeger--Gromoll type metric. For example, M. Benyounes, E. Loubeau and C. M. Wood \cite{blw2} considered $h_{p,q}$ metrics in the context of harmonic maps, whereas the first named author and Sz. M. Walczak \cite{kw} in the context of Riemanian submersions and Gromov--Hausdorff topology.

In \cite{kn}, the authors studied the conformality of a differential $\Phi=\varphi_{\ast}:(TM,\tilh)\to (TN,h)$ of a map $\varphi:M\to N$ between Riemanian manifolds, where $\tilh=h_{p,q,\alpha}$ and $h=h_{r,s,\beta}$ are metrics of Cheeger--Gromoll type. In this paper, we continue considerations concerning the differential of a map. We give necessary and sufficient conditions for a differential to be horizontally conformal and as a corollary we characterize the differential as a harmonic morphisms. 

The idea is the following. The vertical part of Cheeger--Gromoll type metrics is nonlinear with respect to the base point, excluding Sasaki metric. Hence, conformal change of a metric on $M$ does not effect conformal change of Cheeger--Gromoll type metric. Scaling the base point, the condition of horizontal conformality reduces to vanishing of a certain polynomial, which gives restrictions to coefficients $p,q,r,s$ and $\alpha,\beta$. That is why, for horizontal conformality of $\Phi$ there are not many natural choices of $\tilh$ and $h$ (Theorem \ref{maint}), whereas in the context of a harmonic morphism the only possibility is Sasaki metric (Theorem \ref{hmt}). 

\subsection{Preliminary results}
Let $(M,g)$ be a Riemannian manifold and $\pi:TM\to M$ its tangent bundle. The Levi-Civita connection and the projection $\pi$ gives the natural splitting $TTM=\mathcal{H}\oplus \mathcal{V}$ of the second tangent bundle $\pi_{\ast}:TTM\to TM$, where the vertical distribution $\mathcal{V}$ is the kernel of $\pi_{\ast}$ and the horizontal distribution $\mathcal{H}$ is the kernel of the conection map $K$. If $X,\xi\in T_xM$ then there is a unique vertical vector $X^v_{\xi}$ and a unique horizontal vector $X^h_{\xi}$ in $T_{\xi}TM$ such that $\pi_{\ast}X^h_{\xi}=X$ and $KX^v_{\xi}=X$. Moreover, any $A\in T_{\xi}TM$ has a unique decomposition into horozontal and vertical part, $A=\mathcal{V}A+\mathcal{H}A$. For more details on decomposition of the tangent bundle see \cite{dom}.

Let $p,q,\alpha$ be constants, $q$ non--negative, $\alpha$ positive. Define $(p,q,\alpha)$-metric $h=h_{p,q,\alpha}$ on $TM$ as follows. For every $A,B\in T_{\xi}(TM)$,
\[
h(A,B)=g(\pi_\ast A,\pi_\ast B)+\omega_{\alpha}(\xi)^{p}\big(g(KA,KB)+qg(KA,\xi)g(KB,\xi)\big),
\]
where  $\omega_{\alpha}(\xi)=(1+\alpha g(\xi,\xi))^{-1}$. The Riemannian metric $h_{p,q,\alpha}$ is a generalization of the metric considered in \cite{blw1,blw2} and is a special case of a metric considered in \cite{mun}. In particular, $h_{0,0,\alpha}$ (or $h_{p,0,0}$) is Sasaki metric, $h_{1,1,1}$ Cheeger-Gromoll metric.

We will often write $\skal{\cdot}{\cdot}_M$ for $g$ and $\skal{\cdot}{\cdot}_{TM}$ for $h=h_{p,q,\alpha}$. The lenght of a vector will be denoted by $|\cdot|_M$ and $|\cdot|_{TM}$, respecitively.

Consider now a smooth map $\varphi:M\to N$ between Riemannian manifolds $M$ and $N$. Let $\varphi^{-1}TN\to M$ be a pull--back bundle. There is a unique connection $\nabla^{\varphi}$ in this bundle characterised by the property \cite{bw}
\[
\nabla^{\varphi}_X(Y\circ\varphi)=\nabla^N_{\varphi_{\ast}X} Y,\quad X\in T_xM, Y\in\Gamma(TN).
\]
Then we easily obtain
\begin{equation} \label{e1}
\nabla^{\varphi}_X\varphi_{\ast}Y-\nabla^{\varphi}_Y\varphi_{\ast}X=\varphi_{\ast}[X,Y],\quad X,Y\in\Gamma(TM).
\end{equation} 
The second fundamental form of $\varphi$ is $B=\nabla\varphi_{\ast}$,
\begin{equation} \label{e2}
B(X,Y)=\nabla^{\varphi}_X \varphi_{\ast}Y-\varphi_{\ast}(\nabla^M_XY), \quad X,Y\in \Gamma(TM).
\end{equation}
By \eqref{e1}, we get that $B$ is symmetric and hence tensorial in both variables. If $B=0$ we say that $\varphi$ is {\it totally geodesic}. We will need the following lemma.
\begin{lemma} \label{l1} 
For $X,\xi\in T_xM$
\begin{align}
\varphi_{\ast\ast}X^v_{\xi} &=(\varphi_{\ast}X)^v_{\varphi_{\ast}\xi}, \label{l1e1} \\
\varphi_{\ast\ast}X^h_{\xi} &=(\varphi_{\ast}X)^h_{\varphi_{\ast}\xi}+(B(X,\xi))^v_{\varphi_{\ast}\xi}. \label{l1e2}
\end{align}
\end{lemma}
\begin{proof} Follows by the definition of a connection map and vertical and horizontal distributions. Details are left to the reader.
\end{proof}

Let $\varphi:(M,\skal{\cdot}{\cdot}_M)\to (N,\skal{\cdot}{\cdot}_N)$ be a smooth map between Riemannian manifolds, $\dim M>\dim N$. Let $\mathcal{V}^{\varphi}={\rm ker}\varphi_{\ast}$ be the vertical distribution and $\mathcal{H}^{\varphi}=(\mathcal{V}^{\varphi})^{\bot}$, the orthogonal complement of $\mathcal{H}^{\varphi}$ with respect to $\skal{\cdot}{\cdot}_M$, the horozontal distribution, $TM=\mathcal{V}^{\varphi}\oplus\mathcal{H}^{\varphi}$. Each $X\in T_xM$ has therefore a unique decomposition
\[
X=X^{\top}+X^{\bot}
\]
into vertical and horizontal part.

We say that $\varphi$ is {\it horizontally conformal} if for every $x\in M$ either $\varphi_{\ast x}=0$ or $\varphi_{\ast x}:\mathcal{H}^{\varphi}_x\to T_{\varphi(x)}N$ is surjective and
\[
\skal{\varphi_{\ast x}X}{\varphi_{\ast x}Y}_N=\lambda(x)\skal{X}{Y}_M,\quad X,Y\in\mathcal{H}^{\varphi}_x,
\]
where $\lambda(x)$ is positive. We call $\lambda$ the {\it dilatation} of $\varphi$. Let $C_{\varphi}$ denote the set of {\it critical points} i.e., points $x\in M$ such that $\varphi_{\ast x}=0$. One can easily prove

\begin{lemma} \label{l2}
For $\Phi=\varphi_{\ast}:TM\to TN$ we have $\pi(C_{\Phi})\subset C_{\varphi}$.
\end{lemma}

Moreover, by Lemma \ref{l1} we have
\begin{lemma} \label{lsub}
If a map $\varphi:M\to N$ is a submersion, i.e., surjective map of maximal rank, then so is $\Phi=\varphi_{\ast}:TM\to TN$.
\end{lemma}

\begin{remark} \label{rem1}
By Lemmas \ref{l2} and \ref{lsub}, if $\varphi:M\to N$ and $\Phi=\varphi_{\ast}:TM\to TN$ are horizontally conformal, then the differential of horizontally conformal submersion $\varphi:M\setminus C_{\varphi}\to N$ is a horizontaly conformal submersion $\Phi:TM\setminus \pi^{-1}(C_{\varphi})\to TN$. Hence, throughout the paper, without loss of generality, we may assume that horizontaly conformal map $\varphi$ is a submersion i.e. the set of critical points is empty.
\end{remark}

\begin{theorem} \label{tPhi}
Assume $\varphi:M\to N$ is a submersion. Let $\Phi=\varphi_{\ast}:TM\to TN$ and equip tangent bundles $TM$ and $TN$ with Cheeger--Gromoll type metrics $h_{p,q,\alpha}$ and $h_{r,s,\beta}$, respectively. Then, with respect to submersion $\Phi$, the second tangent bundle splits into orthogonal sum $T_{\xi}TM=\mathcal{V}^{\Phi}_{\xi}\oplus\mathcal{H}^{\Phi}_{\xi}$, where
\begin{equation} \label{tPhie1}
\mathcal{V}^{\Phi}_{\xi}={\rm Span}\{\eta^h_{\xi}+(\nabla^M_{\xi}\eta)^v_{\xi} \mid \eta\in\Gamma(\mathcal{V}^{\varphi}) \}, 
\end{equation}
and
\begin{equation} \label{tPhie2}
\mathcal{H}^{\Phi}_{\xi}={\rm Span}\{ X^v_{\xi}-q\omega_q(\xi)\skal{X}{\xi}\xi^v_{\xi}+\omega_{\alpha}(\xi)^p(\nabla^M_{\xi}X)^h_{\xi} \mid X\in\Gamma(\mathcal{H}^{\varphi}) \}. 
\end{equation}
\end{theorem}
\begin{proof} Let $\varphi:(M,g_M)\to (N,g_N)$ be a submersion. Fix $\xi\in T_xM$, $x\in M$. Denote the right hand side of \eqref{tPhie1} by $\mathbb{V}$ and the right hand side of \eqref{tPhie2} by $\mathbb{H}$. By Lemma \ref{l1}, for $\eta\in\Gamma(\mathcal{V}^{\varphi})$ 
\[
\Phi_{\ast}(\eta^h_{\xi}+(\nabla^M_{\xi}\eta)^v_{\xi})=(\varphi_{\ast}\eta)^h_{\varphi_{\ast}\xi}+B(\eta,\xi)^v_{\varphi_{\ast}\xi}+(\varphi_{\ast}\nabla^M_{\xi}\eta)^v_{\varphi_{\ast}\xi}=0.
\]
Hence, $\mathbb{V}\subset \mathcal{V}^{\Phi}_{\xi}$. Let $f$ be a smooth function such that $\xi f=1$ and $f(x)=0$. Then
\[
\mathcal{V}^{\Phi}_{\xi}\ni (f\eta)^h_{\xi}+(\nabla^M_{\xi}f\eta)^v_{\xi}=\eta^v_{\xi}.
\]
Therefore, we see that $\dim \mathbb{V}=2(m-n)$, where $m=\dim M$ and $n=\dim N$. Thus $\mathcal{V}^{\Phi}_{\xi}=\mathbb{V}$, so \eqref{tPhie1} holds. Now, let $\eta\in\Gamma(\mathcal{V}^{\varphi})$ and $X\in\Gamma(\mathcal{H}^{\varphi})$. For $A=\eta^h_{\xi}+(\nabla^M_{\xi}\eta)^v_{\xi}$ and $B=X^v_{\xi}-q\omega_q(\xi)\skal{X}{\xi}\xi^v_{\xi}+\omega_{\alpha}(\xi)^p(\nabla^M_{\xi}X)^h_{\xi}$
\begin{gather*}
\begin{split}
\skal{A}{B}_{TM} &=\omega_{\alpha}(\xi)^p\bigg( \skal{X}{\nabla^M_{\xi}\eta}_M+q\skal{X}{\xi}_M\skal{\xi}{\nabla^M_{\xi}\eta}_M \\
&-q\omega_q(\xi)\skal{X}{\xi}_M\skal{\nabla^M_{\xi}\eta}{\xi}_M \\
&-q^2\omega_q(\xi)\skal{X}{\xi}_M\skal{\nabla^M_{\xi}\eta}{\xi}_M|\xi|^2_M+\skal{\nabla^M_{\xi}X}{\eta}_M \bigg) \\
&=q\omega_{\alpha}(\xi)^p\skal{X}{\xi}_M\skal{\nabla^M_{\xi}\eta}{\xi}_M(1-\omega_q(\xi)-q\omega_q(\xi)|\xi|^2_M) \\
&=0.
\end{split}
\end{gather*}
Thus $\mathbb{H}$ is orthogonal to $\mathcal{V}^{\Phi}_{\xi}$. Again, as above, taking a function $f$ such that $\xi f=1$ and $f(x)=0$, we get $X^h_{\xi}\in\mathbb{H}$ for $X\in\Gamma(\mathcal{H}^{\varphi})$. Hence, $\dim \mathbb{H}=2m$. Therefore, $\mathbb{H}=\mathcal{H}^{\Phi}_{\xi}$ and \eqref{tPhie2} holds. 
\end{proof}

The proof of Theorem \ref{tPhi} also implies
\begin{corollary} \label{cPhi}
Assume $\varphi:M\to N$ is a submersion. Let $\Phi=\varphi_{\ast}:TM\to TN$ and equip tangent bundles $TM$ and $TN$ with Cheeger--Gromoll type metrics $h_{p,q,\alpha}$ and $h_{r,s,\beta}$, respectively. Then
\begin{align}
\eta^v_{\xi} &\in\mathcal{V}^{\Phi}_{\xi} \quad\textrm{for $\eta\in\mathcal{V}^{\varphi}$}, \label{cPhie1} \\
X^h_{\xi} &\in\mathcal{H}^{\Phi}_{\xi} \quad\textrm{for $X\in\mathcal{H}^{\varphi}$}. \label{cPhie2}
\end{align}
\end{corollary}

Let $\varphi:M\to N,$ be a Riemannian submersion, $\Phi=\varphi_{\ast}:TM\to TN$. A vector field $\hat Z\in \Gamma(\mathcal{H}^{\varphi})$ is called {\it basic} if there is a vector field $Z\in \Gamma(TN)$ such that $\varphi_{\ast}\hat Z=Z\circ\varphi$. The correspondence $Z\mapsto \hat Z$ is one--to--one, since $\varphi_{\ast x}:\mathcal{H}^{\varphi}_x\to T_{\varphi(x)}N$ is an isomorphisms. Let $T:\Gamma(\mathcal{H}^{\varphi})\times\Gamma(\mathcal{H}^{\varphi})\to\Gamma(\mathcal{V}^{\varphi})$ be an itegrability tensor of $\varphi$,
\[
T(X,Y)=\frac{1}{2}[X,Y]^{\top}.
\]
 Let $\nabla^M$ and $\nabla^N$ be the Levi--Civita conections of $\skal{\cdot}{\cdot}_M$ and $\skal{\cdot}{\cdot}_N$, respectively. Then one can prove that for basic vector fields $\hat Z, \hat W$ and vertical vector field $\xi$ 
\begin{equation} \label{e12}
-\skal{(\nabla^M_{\xi}\hat Z)^{\bot}}{\hat W}_M=\skal{\xi}{T(\hat Z,\hat W)}_M.
\end{equation}
For more details on Riemannian submersions see \cite[Chapter 9]{bes}.

Assume now, $\varphi:(M,\skal{\cdot}{\cdot}_M)\to (N,\skal{\cdot}{\cdot}_N)$ is a horizontally conformal map with a dilatation $\lambda$. If we put $\skal{\cdot}{\cdot}_{\lambda}=\lambda\skal{\cdot}{\cdot}_M$, then $\varphi:(M,\skal{\cdot}{\cdot}_{\lambda})\to (N,\skal{\cdot}{\cdot}_N)$ is a Riemannian submersion. Clearly, horizontal and vertical distrbutions with respect to $\skal{\cdot}{\cdot}_M$ and $\skal{\cdot}{\cdot}_{\lambda}$ coincide. The Levi--Civita connections $\nabla^M$ and $\nabla^{\lambda}$ of $\skal{\cdot}{\cdot}_M$ and $\skal{\cdot}{\cdot}_{\lambda}$ satisfy 
\begin{equation} \label{confd}
\nabla^{\lambda}_XY=\nabla^M_XY+S(X,Y),
\end{equation}
where $S$ is a symmetric tensor field given by
\[
S(X,Y)=\frac{1}{2\lambda}((X\lambda)Y+(Y\lambda)X-\skal{X}{Y}_M{\rm grad\,\lambda}).
\]
Let $\Phi=\varphi_{\ast}:TM\to TN$ and equip $TM$ and $TN$ with Cheeger--Gromoll type metrics $h_{p,q}$ and $h_{r,s}$, respectively. For simplicity, put
\[
P(X,\xi)=\sum_{i=1}^n\skal{\xi}{T(X,e_i)}_M\varphi_{\ast}e_i,\quad \xi\in T_xM, X\in\mathcal{H}^{\varphi}_x,
\]
where $e_1,\ldots,e_n$ is an orthonormal basis of $\mathcal{H}^{\varphi}_x$. 
\begin{lemma} \label{lmain}
Let $X\in\mathcal{H}^{\varphi}_x$, $\xi\in T_xM$, $x\in M$. Then
\begin{equation} \label{emain}
\Phi_{\ast}(X^h_{\xi})=(\varphi_{\ast}X)^h_{\varphi_{\ast}\xi}+(\varphi_{\ast}S(X,\xi))^v_{\varphi_{\ast}\xi}+P(X,\xi)^v_{\varphi_{\ast}\xi}.
\end{equation}
\end{lemma}
\begin{proof}
We may assume that $X^h_{\hat\zeta}=(\gamma^h)^{\cdot}(0)$, where $\gamma$ is a curve on $M$ such that $\gamma(0)=x$, $\dot{\gamma}(t)\in\mathcal{H}^{\varphi}$ and $\dot{\gamma}(0)=X$, and $\gamma^h$ is a horizontal lift of $\gamma$ to $TM$ such that $\gamma^h(0)=\xi$. Put $\eta=\gamma^h$ for simplicity. Then
\[
\nabla^N_{\varphi_{\ast}\dot{\gamma}}\varphi_{\ast}\eta=\nabla^N_{\varphi_{\ast}\dot{\gamma}}\varphi_{\ast}\eta^{\bot}=\varphi_{\ast}(\nabla^{\lambda}_{\dot{\gamma}}\eta^{\bot})
\]
and by \eqref{confd} and the fact that $\eta$ is parallel with respect to $\nabla^M$,
\[
\nabla^{\lambda}_{\dot{\gamma}}\eta=S(\dot{\gamma},\eta).
\]
Therefore,
\[
\nabla^{\lambda}_{\dot{\gamma}}\eta^{\bot}=S(\dot{\gamma},\eta)-\nabla^{\lambda}_{\dot{\gamma}}\eta^{\top}.
\]
Next, for a vector field $Y$ on $N$, by \eqref{e12}
\[
\skal{\varphi_{\ast}(\nabla^{\lambda}_{\dot{\gamma}}\eta^{\top})}{Y}_N=\lambda\skal{\nabla^{\lambda}_{\dot{\gamma}}\eta^{\top}}{\hat Y}_M =\lambda\skal{\nabla^{\lambda}_{\eta^{\top}}\dot{\gamma}}{\hat Y}_M=-\lambda\skal{\eta^{\top}}{T(\dot{\gamma},\hat Y)}_M.
\]
Thus
\[
\varphi_{\ast}(\nabla^{\lambda}_{\dot{\gamma}}\eta^{\top})=\frac{1}{\lambda}\sum_i \skal{\varphi_{\ast}(\nabla^{\lambda}_{\dot{\gamma}}\eta^{\top})}{\varphi_{\ast}e_i}_N\varphi_{\ast}e_i=-\lambda\sum_i \skal{\eta^{\top}}{T(\dot{\gamma},e_i)}_M\varphi_{\ast}e_i.
\]
Since
\begin{gather*}
\begin{split}
K(\Phi_{\ast}X^h_{\xi}) &=\nabla^N_{\varphi_{\ast}X}\varphi_{\ast}\eta \\
&=\varphi_{\ast}(\nabla^{\lambda}_{\dot{\gamma}}\eta^{\bot}) \\
&=\varphi_{\ast}S(X,\eta)-\varphi_{\ast}(\nabla^{\lambda}_{\dot{\gamma}}\eta^{\top}) \\
&=\varphi_{\ast}S(X,\eta)+\sum_i \skal{\eta^{\top}}{T(X,e_i)}_M\varphi_{\ast}e_i \\
&=\varphi_{\ast}S(X,\xi)+\sum_i \skal{\xi}{T(X,e_i)}_M\varphi_{\ast}e_i
\end{split}
\end{gather*}
and
\[
\pi_{\ast}(\Phi_{\ast}X^h_{\xi})=\varphi_{\ast}(\pi_{\ast}X^h_{\xi})=\varphi_{\ast}X,
\]
the equality \eqref{emain} holds.
\end{proof}

By Lemma \ref{lmain} we have (see also \cite[Lemma 4.5.1]{bw}).
\begin{corollary} \label{clmain}
\begin{enumerate}
\item[$(1)$] Let $X,Y\in\mathcal{H}^{\varphi}_x$, $\xi\in T_xM$. Then
\begin{align}
B(X,\xi) &=\varphi_{\ast}(S(X,\xi))+P(X,\xi), \label{cme1} \\
\skal{B(X,\xi)}{\varphi_{\ast}Y}_N &=\lambda(\skal{S(X,\xi)}{Y}_M+\skal{\xi}{T(X,Y)}_M). \label{cme2}
\end{align}
\item[$(2)$] Let $X\in\mathcal{H}^{\varphi}_x$, $\xi\in \mathcal{V}^{\varphi}_x$. Then
\begin{align} 
\skal{\varphi_{\ast}(S(X,\xi))}{P(X,\xi)}_N &=0, \label{cme3} \\
|B(X,\xi)^v_0|^2_{TN} &=|(\varphi_{\ast}(S(X,\xi)))^v_0|^2_{TN}+|P(X,\xi)^v_0|^2_{TN}. \label{cme4}
\end{align}
\item[$(3)$] $B(X,\xi)=0$ for all $X\in\Gamma(\mathcal{H}^{\varphi})$, $\xi\in\Gamma(TM)$ if and only if $T=0$ and $S=0$.
\end{enumerate}
\end{corollary}
\begin{proof}
\eqref{cme1} follows by \eqref{emain} and \eqref{l1e2}. For $X,Y\in\mathcal{H}^{\varphi}_x$ and $\xi\in T_xM$ we have
\begin{gather*}
\begin{split}
\skal{B(X,\xi)}{\varphi_{\ast}Y}_N &=\skal{\varphi_{\ast}(S(X,\xi))}{\varphi_{\ast}Y}_N+\sum_i\skal{\xi}{T(X,e_i)}_M\skal{\varphi_{\ast}e_i}{\varphi_{\ast} Y}_N \\
&=\lambda\skal{S(X,\xi)}{Y}_M+\lambda\skal{\xi}{T(X,Y)}_M,
\end{split}
\end{gather*}
which proves \eqref{cme2}. Now, let $X\in\mathcal{H}^{\varphi}_x$ and $\xi\in\mathcal{V}^{\varphi}_x$. Then
\[
\varphi_{\ast}(S(X,\xi))=\frac{1}{2\lambda}(\xi \lambda)\varphi_{\ast}X.
\]
Thus
\begin{gather*}
\begin{split}
\skal{\varphi_{\ast}(S(X,\xi))}{P(X,\xi)}_N &=\frac{1}{2\lambda}(\xi\lambda)\sum_i\skal{\xi}{T(X,e_i)}_M \skal{\varphi_{\ast}X}{\varphi_{\ast}e_i}_N \\
&=\frac{1}{2}\skal{\xi}{T(X,X)}_M=0,
\end{split}
\end{gather*}
since $T$ is skew--symmetric. Hence \eqref{cme3} holds, which implies \eqref{cme4}. $(3)$ is a consequence of \eqref{cme1} and \eqref{cme3}.
\end{proof}

Let $X\in\mathcal{H}^{\varphi}_x$ and $\xi\in T_xM$. Decompose $S(X,\xi)^v_{\xi}$ into horizontal and vertical part with respect to $\mathcal{V}^{\Phi}$ and $\mathcal{H}^{\Phi}$,
\[
S(X,\xi)^v_{\xi}=S^{\top}_{\Phi}(X,\xi)+S^{\bot}_{\Phi}(X,\xi).
\]
Equip tangent bundles $TM$ and $TN$ with Cheeger--Gromoll type metrics $h_{p,q,\alpha}$ and $h_{r,s,\beta}$, respectively. 
\begin{lemma} \label{l5}
Assume $\varphi$ and $\Phi$ are both horizontally conformal with dilatations $\lambda$ and $\Lambda$, respectively. Let $X\in\mathcal{H}^{\varphi}_x$, $\xi\in T_xM$. Then
\begin{align}
\Lambda(\xi)|X|_M^2 &=\lambda(x)|X|_M^2+|B(X,\xi)^v_{\varphi_{\ast}\xi}|_{TN}^2, \label{e13} \\
\Lambda(\xi)|X|_M^2+\Lambda(\xi)|S^{\bot}_{\Phi}(X,\xi)|_{TN}^2 &=\lambda(x)|X|_M^2+|P(X,\xi)^v_{\varphi_{\ast}\xi}|_{TN}^2. \label{e14}
\end{align} 
\end{lemma}
\begin{proof}
\eqref{e13} follows directly by \eqref{l1e2} (and Corollary \ref{cPhi}). To prove \eqref{e14}, define $V=X^h_{\xi}-S(X,\xi)^v_{\varphi_{\ast}\xi}$. Then
\[
V+S^{\top}_{\Phi}(X,\xi)=X^h_{\xi}-S^{\bot}_{\Phi}(X,\xi)\in\mathcal{H}^{\Phi}.
\]
Therefore,
\begin{gather*}
\begin{split}
|\Phi_{\ast}(X^h_{\xi}-S^{\bot}_{\Phi}(X,\xi))|_{TN}^2 &=\Lambda(\xi)|X^h_{\xi}-S^{\bot}_{\Phi}(X,\xi)|_{TN}^2 \\
&=\Lambda(\xi)|X|_M^2+\Lambda(\xi)|S^{\bot}_{\Phi}(X,\xi)|_{TN}^2,
\end{split}
\end{gather*}
since $\skal{X^h_{\xi}}{S^{\bot}_{\Phi}(X,\xi)}_{TM}=\skal{X^h_{\xi}}{S(X,\xi)^v_{\xi}}=0$, and, by \eqref{emain},
\begin{gather*}
 \begin{split}
|\Phi_{\ast}(V+S^{\top}_{\Phi}(X,\xi))|_{TN}^2 &=|\Phi_{\ast}V|_{TN} \\
&=|(\phi_{\ast}X)^h_{\varphi_{\ast}\xi}+P(X,\xi)^v_{\varphi_{\ast}\xi}|_{TN}^2 \\
&=\lambda(x)|X|_M^2+|P(X,\xi)^v_{\varphi_{\ast}\xi}|_{TN}^2.
\end{split}
\end{gather*}
Hence \eqref{e14} holds.
\end{proof}

\begin{lemma} \label{lconst}
If $\Phi$ and $\varphi$ are both horizontally conformal, then $\lambda$ is constant.
\end{lemma}
\begin{proof}
Let $X\in\mathcal{H}^{\varphi}$ and $\xi\in\mathcal{V}^{\varphi}$. Comparing \eqref{e13} and \eqref{e14} and using \eqref{cme4} we get
\[
|S^{\bot}_{\Phi}(X,\xi)|^2_{TN}+|(\varphi_{\ast}(S(X,\xi)))^v_0|^2=0,
\]
which implies, $\varphi_{\ast}(S(X,\xi))=0$, so $\xi\lambda=0$. It follows that ${\rm grad}\,\lambda$ is horizontal.

Assume now, $\xi\in\mathcal{H}^{\varphi}$. Then $P(X,\xi)=0$. Again, comparing \eqref{e13} and \eqref{e14} we get
\[
0=\Lambda(\xi)|S^{\bot}_{\Phi}(X,\xi)|_{TN}^2+|B(X,\xi)^v_{\varphi_{\ast}\xi}|_{TN}^2.
\]
Thus
\[
B(X,\xi)=0\quad\textrm{and}\quad S^{\bot}_{\Phi}(X,\xi)=0.
\]
Hence, as before, $\varphi_{\ast}S(X,\xi)=0$. Since ${\rm grad}\,\lambda$ is horizontal, we may put $X=\xi={\rm grad}\,\lambda$. Then
\[
0=\varphi_{\ast}S({\rm grad}\,\lambda,{\rm grad}\,\lambda)=-\frac{1}{2\lambda}|{\rm grad}\,\lambda|^2_M\varphi_{\ast}({\rm grad}\,\lambda).
\]
Therefore, ${\rm grad}\,\lambda$ is vertical. Finally, $\lambda$ is constant.
\end{proof}

\begin{lemma} \label{l4}
Assume $\dim N\geq 2$. Let $\Phi$ be a horizontally conformal map with dilatation $\Lambda$. Then $\varphi$ is horizontally conformal with constant dilatation $\lambda=\Lambda(0)$. Moreover, $\Lambda$ is constant and the horizontal distribution $\mathcal{H}^{\varphi}$ is integrable.
\end{lemma}
\begin{proof}
Let $X,Y\in\mathcal{H}^{\varphi}_x$, $0=0_x\in T_xM$. Then by Corollary \ref{cPhi}, $X^h_0,Y^h_0\in\mathcal{H}^{\Phi}_0$. By \eqref{l1e2}, $\Phi_{\ast}(X^h_0)=(\varphi_{\ast}X)^h_0$ and $\Phi_{\ast}(Y^h_0)=(\varphi_{\ast}Y)^h_0$. Thus
\[
\Lambda(0)\skal{X}{Y}_M=\skal{\Phi_{\ast}X^h_0}{\Phi_{\ast}Y^h_0}_{TN}=\skal{\varphi_{\ast}X}{\varphi_{\ast}Y}_N.
\]
Hence $\varphi$ is conformal with dilatation $\lambda(x)=\Lambda(0_x)$. By Lemma \ref{lconst} $\lambda$ is constant. Thus $S=0$. 

Let $X,Y\in\Gamma(\mathcal{H}^{\varphi})$, $\xi\in\mathcal{V}^{\varphi}$. Then $X^v_{\xi}+\omega_{\alpha}(\xi)^p(\nabla_{\xi}X)^h_{\xi}, Y^h_{\xi}\in\mathcal{H}^{\Phi}_{\xi}$ by Theorem \ref{tPhi} and Corollary \ref{cPhi}. Moreover, by Corollary \ref{clmain}
\[
\skal{B(Y,\xi)}{\varphi_{\ast}X}_N=\lambda\skal{\xi}{T(Y,X)}_M.
\]
Hence, by horizontal conformality of $\Phi$ and \eqref{e12} we get
\[
(\Lambda(\xi)-\lambda-\lambda(1+\alpha|\xi|^2)^p)\skal{\xi}{T(Y,X)}_M=\skal{B(\nabla^M_{\xi}X,\xi)}{B(Y,\xi)}_N.
\]
Put $A(\xi)=\Lambda(\xi)-\lambda$. Then by \eqref{e13}, $A(t\xi)=t^2 A(\xi)$. Hence, replacing $\xi$ by $t\xi$ and assuming for simplicity $|\xi|_M=1$, we get
\[
t^2(A(\xi)\skal{\xi}{T(Y,X)}_M-\skal{B(\nabla^M_{\xi}X,\xi)}{B(Y,\xi)}_M)=\lambda(1+\alpha t^2)^p\skal{\xi}{T(Y,X)}_M.
\]
Therefore, $\skal{\xi}{T(Y,X)}_M=0$ and, since $\xi\in\Gamma(\mathcal{V}^{\varphi})$ was taken arbitrary, $T=0$. This implies, together with the fact that $S=0$ and Corollary \ref{clmain} that
\[
B(X,\xi)=0\quad \textrm{for $X\in\mathcal{H}^{\varphi}$ and $\xi\in TM$}.
\]
Hence, by \eqref{e13}, $\Lambda(\xi)=\lambda$ for any $\xi$.
\end{proof}

\section{Main results}
\subsection{Differential as a horizontally confomal map}
Consider a smooth map $\varphi:(M,\skal{\cdot}{\cdot}_M)\to (N,\skal{\cdot}{\cdot}_N)$, $\dim M>\dim N\geq 2$, between Riemanian manifolds. Let $\Phi=\varphi_{\ast}:(TM,\tilh)\to (TN,h)$, where $\tilh=h_{p,q,\alpha}$ and $h=h_{r,s,\beta}$ are Cheeger--Gromoll type metrics.
\begin{theorem} \label{maint}
$\Phi$ is horizontally conformal if and only if $\varphi$ is totally geodesic and horizontally conformal with constant dilatation $\lambda$, $p\alpha=r\beta=0$ and $q=\lambda s$. Then, the dilatation $\Lambda$ of $\Phi$ is constant and equal to $\lambda$, horizontal distribution $\mathcal{H}^{\varphi}$ is integrable and vertical distribution $\mathcal{V}^{\varphi}$ is totally geodesic.
\end{theorem}
\begin{proof}
Assume $\Phi$ is horizontally conformal. Then by Lemma \ref{l4}, $\varphi$ is horizontally conformal with constant dilatation $\lambda$, the dilatation $\Lambda$ of $\Phi$ is constant and equal to $\lambda$ and by Corollary \ref{clmain}, $B(X,\xi)=0$ for $X\in \mathcal{H}^{\varphi}$ and $\xi\in TM$. 

Let $X\in\Gamma(\mathcal{H}^{\varphi})$ and $\xi\in\Gamma(\mathcal{V}^{\varphi})$. Then $A=X^v_{\xi}+\omega_{\alpha}(\xi)^p(\nabla^M_{\xi}X)^h_{\xi}\in \mathcal{H}^{\Phi}_{\xi}$ by Theorem \ref{tPhi}. Since $T=0$, by \eqref{e12} $\nabla^M_{\xi}X\in\mathcal{V}^{\varphi}$. Then $|\Phi_{\ast}A|^2_{TN}=\lambda|A|^2_{TM}$ implies
\begin{multline*}
\lambda|X|^2_M(1+\alpha|\xi|^2_M)^p(1-(1+\alpha|\xi|^2_M)^p)+\lambda|\nabla^M_{\xi}X|^2_M \\
=|B(\nabla^M_{\xi}X,\xi)|^2_N+2(1+\alpha|\xi|^2_M)^p\skal{B(\nabla^M_{\xi}X,\xi)}{\varphi_{\ast}X}_N.
\end{multline*}
Replacing $\xi$ by $t\xi$, assuming $|\xi|^2_M=1$ and computing the second derivative at $t=0$, we get
\[
-\alpha p\lambda|X|^2_M+\lambda|\nabla^M_{\xi}X|^2_M=2\skal{B(\nabla^M_{\xi}X,\xi)}{\varphi_{\ast}X}_N.
\]
Since $\skal{B(\nabla_{\xi}X,\xi)}{\varphi_{\ast}X}_N=\lambda|\nabla^M_{\xi}X|^2_M$, we obtain $-p\alpha|X|^2_M=|\nabla^M_{\xi}X|^2_M$. Therefore, $p\alpha=0$ and $\nabla^M_{\xi}X=0$. This implies $B=0$ globally, so $\varphi$ is totally geodesic. In particular, horizontal distrubution is integrable and vertical distribution is totally geodesic.

Moreover, for $X\in\Gamma(\mathcal{H}^{\varphi})$ basic and $\xi\in\mathcal{H}^{\varphi}_x$ such that $\xi$ and $X_x$ are orthogonal, by Theorem \ref{tPhi} and the fact that $\alpha p=0$ we have $A=X^v_{\xi}+(\nabla^M_{\xi}X)^h_{\xi}\in \mathcal{H}^{\Phi}_{\xi}$. Since $\nabla^M_{\xi}X\in\mathcal{H}^{\varphi}_x$, the condition $|\Phi_{\ast}A|^2_{TN}=\lambda|A|^2_{TM}$ simplfies to
\[
\omega_{\beta}(\varphi_{\ast}\xi)^r=1.
\]
This implies $r\beta=0$. Now, for arbitrary $X,\xi\in\Gamma(\mathcal{H}^{\varphi})$, $A=X^v_{\xi}-q\omega_q(\xi)\skal{X}{\xi}\xi^v_{\xi}+(\nabla^M_{\xi}X)^h_{\xi}\in\mathcal{H}^{\Phi}$. Then, as above by horizontal conformality of $\Phi$, after some computations
\[
q+q^3\omega_q(\xi)^2|\xi|^4_M-2q^2\omega_q(\xi)|\xi|^2_M =\lambda s+\lambda s q^2\omega_q(\xi)^2|\xi|^2_M-2\lambda s q\omega_q(\xi)|\xi|^2_M.
\]
Hence $q=\lambda s$.

Conversely, assume $\varphi$ is totally geodesic and horizontally conformal with constant dilatation $\lambda$, $p\alpha=r\beta=0$ and $q=\lambda s$. By Theorem \ref{tPhi} $\Phi:TM\to TN$ is a submersion and vertical and horizontal distributions $\mathcal{V}^{\Phi}$ and $\mathcal{H}^{\Phi}$ are given by \eqref{tPhie1} and \eqref{tPhie2}, respectively. It remains to show that $\Phi_{\ast}:\mathcal{H}^{\Phi}\to T(TN)$ is conformal, but this follows immediately by simple calculations.
\end{proof}

\subsection{Differential as a harmonic morphism}
Let $\varphi:(M,\skal{\cdot}{\cdot}_M)\to (N,\skal{\cdot}{\cdot}_N)$. Consider notation from the first section. We say that $\varphi$ is {\it harmonic} if its tension field $\tau(\varphi)={\rm tr}B$ vanishes. If $e_1,\ldots, e_n$ is an orthonormal frame on $M$ then
\begin{equation} \label{e10}
\tau(\varphi)=\sum_i(\nabla^{\varphi}_{e_i}\varphi_{\ast}e_i-\varphi_{\ast}(\nabla^M_{e_i}e_i)).
\end{equation} 

A map $\varphi:M\to N$ is said to be a {\it harmonic morphism} if for every harmonic function $f:G\to\mathbb{R}$ defined on an open subset $G$ of $N$ with $\varphi^{-1}(G)$ nonempty, the composition $f\circ \varphi$ is a harmonic map on $\varphi^{-1}(G)$. There is a useful characterisation of harmonic morphisms due to Fuglede and Ishihara. Namely, see \cite{bw}, $\varphi$ is a harmonic morphisms if and only if it is horizontally conformal and harmonic. Moreover, if $\varphi$ is horizontally conformal, then its tension field simplifies to \cite{bw}
\begin{equation} \label{e11}
\tau(\varphi)=-\frac{n-2}{2}\varphi_{\ast}({\rm grad}({\rm ln}\lambda))-(m-n)\varphi_{\ast}(\kappa_{\varphi}),
\end{equation} 
where $m=\dim M$, $n=\dim N$ and $\kappa_{\varphi}$ is the mean curvature vector field of the fibres of $\varphi$,
\[
\kappa_{\varphi}=\frac{1}{m-n}\sum_i(\nabla^M_{e_i}e_i))^{\bot},
\]
where $e_1,\ldots,e_{m-n}$ is an orthonormal frame of $\mathcal{V}^{\varphi}$.

Equip tangent bundle $TM$ with Cheeger--Gromoll type metric $h_{p,q,\alpha}$. Let $\omega_q(\xi)=(1+q|\xi|^2_M)^{-1}$. The Levi--Civita connection $\nabla^{TM}$ corresponding to $h_{p,q,\alpha}$ evaluated at $\xi$ is given by \cite{mun}
\begin{align*}
\nabla^{TM}_{X^h}Y^h &=(\nabla_XY)^h-\frac{1}{2}(R(X,Y)\xi)^v, \\
\nabla^{TM}_{X^h}Y^v &=(\nabla_XY)^v+\frac{1}{2}\omega(\xi)^p(R(\xi,Y)X)^h, \\
\nabla^{TM}_{X^v}Y^h &=\frac{1}{2}\omega(\xi)^p(R(\xi,X)Y)^h, \\
\begin{split}
\nabla^{TM}_{X^v}Y^v &=-\alpha p\omega(\xi)(\skal{X}{\xi}_MY+\skal{Y}{\xi}_MX)^v \\
&+(\alpha p\omega(\xi)+q)\omega_q(\xi)\skal{X}{Y}_M\xi^v_{\xi}+\alpha pq\omega(\xi)\omega_q(\xi)\skal{X}{\xi}_M\skal{Y}{\xi}_M\xi^v_{\xi}.
\end{split}
\end{align*}

Let $\varphi:(M,\skal{\cdot}{\cdot}_M)\to (N,\skal{\cdot}{\cdot}_N)$, $\dim M>\dim N\geq 2$, be a smooth map between Riemanian manifolds, $\Phi=\varphi_{\ast}:(TM,\tilh)\to (TN,h)$, where $\tilh=h_{p,q,\alpha}$ and $h=h_{r,s,\beta}$ are Cheeger--Gromoll type metrics.
\begin{theorem} \label{hmt}
$\Phi$ is a harmonic morphism if and only if $\varphi$ is a totally geodesic harmonic morphism, the dilatation of $\varphi$ is constant and $\tilh, h$ are Sasaki metrics.
\end{theorem}
\begin{proof}
Assume $\Phi$ is a harmonic morphism. In particular, $\Phi$ is horizontally conformal. Hence, by Theorem \ref{maint}, $\varphi$ is horizontally conformal, the dilatations of $\Phi$ and $\varphi$ are constant, $\Lambda=\lambda$, and $\varphi$ is totally geodesic. Therefore $\varphi$ is harmonic, so $\varphi$ is a harmonic morphism. Thus, by \eqref{e11}, $\kappa_{\varphi}=0$, $\kappa_{\Phi}=0$. Moreover, by Theorem \ref{maint}, $p\alpha=r\beta=0$ and $q=\lambda s$. Let $\xi\in\mathcal{H}^{\varphi}$ and let $e_1,\ldots,e_{m-n}$ be an orthonormal frame for $\mathcal{V}^{\varphi}_{\xi}$, $m=\dim M$, $n=\dim N$. Since vertical distribution $\mathcal{V}^{\varphi}$ is, by Theorem \ref{maint} totally geodesic, by Theorem \ref{tPhi} and Corollary \ref{cPhi} 
\[
E_i=(e_i)^h_{\xi},\quad F_i=(e_i)^v_{\xi},\quad i=1,\ldots,m-n,
\]
is an orthonormal frame for $\mathcal{V}^{\Phi}_{\xi}$. Then
\begin{align*}
\nabla^{TM}_{E_i}E_i &=(\nabla^M_{e_i}e_i)^h_{\xi}, \\
\nabla^{TM}_{F_i}F_i &=q\omega_q(\xi)\xi^v_{\xi}.
\end{align*}
Therefore,
\[
\Phi_{\ast}(\kappa_{\Phi})=\frac{1}{2}(\varphi_{\ast}(\kappa_{\varphi}))^h_{\varphi_{\ast}\xi}+\frac{1}{2}q\omega_q(\xi)(\varphi_{\ast}\xi)^v_{\varphi_{\ast}\xi}.
\]
Since $\Phi_{\ast}(\kappa_{\Phi})=0$ and $\varphi_{\ast}(\kappa_{\varphi})=0$, it follows that $q=0$. Hence, $s=0$ and $\tilh$, $h$ are Sasaki metrics.

Conversely, assume $\varphi$ is a totally geodesic harmonic morphism, the dilatation of $\varphi$ is constant and $\tilh$, $h$ are Sasaki metrics. Then, by Theorem \ref{maint} $\Phi$ is horizontally conformal with constant dilatation. Moreover, $\varphi_{\ast}(\kappa_{\varphi})=0$. Hence, similarly as above,
\[
\Phi_{\ast}(\kappa_{\Phi})=\frac{1}{2}(\varphi_{\ast}(\kappa_{\varphi}))^h=0,
\]
and by \eqref{e11}, $\Phi$ is a harmonic morphism.
\end{proof}

\begin{remark} Condition $p\alpha=0$ for Cheeger--Gromoll metric $h_{p,q,\alpha}$ is equavalent to condition $p=0$ for the metric $h_{p,q}$. It follows that in the case of horizontal conformality of a differential, and in consequence for a differential to be a harmonic morphism, it is sufficient to consider only $h_{p,q}$ metrics introduced in \cite{blw2}.
\end{remark}

\end{document}